\documentclass[11pt,reqno]{amsart}
\usepackage[all]{xy}
\usepackage{amssymb}
\usepackage{amsthm}
\usepackage[normalem]{ulem}
\usepackage{amsmath,mathtools}
\usepackage{amscd,enumitem}
\usepackage{verbatim}
\usepackage{eurosym}
\usepackage{float}
\usepackage{color}
\usepackage{dcolumn}
\usepackage[mathscr]{eucal}
\usepackage[all]{xy}
\usepackage{bbm}
\usepackage[textheight=8.5in, textwidth=6.7in]{geometry}

\usepackage{hyperref}
\hypersetup{
	colorlinks=true,
	citecolor=magenta,
	linkcolor=blue,
	filecolor=blue,      
	urlcolor=blue,
}

\newtheorem*{conj*}{Conjecture}
\newtheorem{theorem}{Theorem}[section]

\theoremstyle{definition}

\theoremstyle{plain}

\newtheorem{lemma}[theorem]{Lemma}

\newtheorem*{theorem*}{Theorem}

\newcommand{\Z}{\mathbb{Z}}

\newcommand{\R}{\mathbb{R}}
\newcommand{\N}{\mathbb{N}}

\newcommand{\SL}{\operatorname{SL}}
\newcommand{\C}{\mathbb{C}}

\newcommand{\re}[1]{\text{Re}\(#1\)}

\renewcommand{\pmod}[1]{\,\,({\rm mod}\,\,{#1})}

\numberwithin{equation}{section}

\newtheoremstyle{example}
  {\topsep}   
  {\topsep}   
  {\normalfont}  
  {0pt}       
  {\bfseries} 
  {.}         
  {5pt plus 1pt minus 1pt} 
  {}          
\theoremstyle{example}
\newtheorem*{example}{Example}
\usepackage{enumitem}

\def\({\left(}
\def\){\right)}

\usepackage{centernot}

\begin{document}
\title{Exact formula for 1-lower run overpartitions}
\author{Lukas Mauth}
\address{Department of Mathematics and Computer Science\\Division of Mathematics\\University of Cologne\\ Weyertal 86-90 \\ 50931 Cologne \\Germany}
\email{lmauth@uni-koeln.de}

\makeatletter
\@namedef{subjclassname@2020}{%
	\textup{2020} Mathematics Subject Classification}
\makeatother


\begin{abstract} 
We are going to show an exact formula for lower $1$-run overpartitions. The generating function is of mixed mock-modular type with an overall weight of $0.$ We will apply an extended version of the classical Circle Method. The approach requires bounding modified Kloosterman sums and Mordell integrals.
\end{abstract}

\maketitle

\section{Introduction and statement of results}
An {\it overpartition} \cite{CorteelLovejoy} of a non-negative integer $n$ is a partition of $n$ in which the final occurrence of a number may be overlined. We denote following \cite{CorteelLovejoy} the number of partitions of $n$ by $p(n)$ and the number of overpartitions of $n$ by $\overline{p}(n).$ It is well known that their generating functions are 

\begin{align*}
P(q) &\coloneqq \sum_{n=0}^{\infty} p(n)q^n = \prod_{n=1}^{\infty} \frac{1}{1-q^n} = \frac{1}{(q;q)_{\infty}}, \\
\overline{P}(q) &\coloneqq \sum_{n=0}^{\infty} \overline{p}(n) q^n = \prod_{n=1}^{\infty} \frac{1+q^n}{1-q^n} = \frac{(-q;q)_{\infty}}{(q;q)_{\infty}},
\end{align*}

\noindent
where for $a \in \C$ and $n \in \N_0 \cup \{\infty\}$ we let $(a)_n\coloneqq (a;q)_n \coloneqq \prod_{k=0}^{n-1} (1-aq^k)$ denote the $q$-Pochhammer symbol. These generating functions are essentially special cases of meromorphic modular forms, so called {\it eta-quotients}. Classical questions are to determine asymptotics or exact formulas for the Fourier coefficents of those modular forms. As mentioned in \cite{BringmannMahlburg} the Hardy$-$Ramanujan Tauberian Theorem \cite{HardyRamanujan} for eta-quotients gives the following well-known asymptotics as $n \rightarrow \infty$

\begin{align*}
p(n) &\sim \frac{1}{4\sqrt{3}n}e^{\pi\sqrt{\frac{2n}{3}}}, \\
\overline{p}(n) &\sim \frac{1}{8n}e^{\pi \sqrt{n}}.
\end{align*}

Later Rademacher perfected the Circle Method developed by Hardy and Ramanujan and obtained an exact formula for $p(n)$ \cite{RademacherExact}. To state Rademacher's result we define the {\it Kloosterman sums} \cite{RademacherExact, RademacherZuckerman} $$A_k(n) := \sum_{\substack{0 \leq h < k \\ \gcd(h,k)=1}} \omega_{h,k}e^{\frac{-2\pi i n h}{k}},$$

\noindent
where $\omega_{h,k}$ is defined in Section $2.$ Let $I_\kappa$ denote the modified Bessel function of order $\kappa$ \cite{Temme}. Rademacher's exact formula for $p(n)$ can then be stated in the following form \cite{RademacherZuckerman}

$$p(n) = \frac{2\pi}{(24n-1)^{3/4}} \sum_{k=1}^{\infty} \frac{A_k(n)}{k} I_{\frac{3}{2}}\left(\frac{\pi\sqrt{24n-1}}{6k}\right).$$

In \cite{Zuckerman} Zuckerman generalized the Circle Method to provide exact formulas for the Fourier coefficients at any cusp for arbitrary weakly holomorphic modular forms of negative weight for finite index subgroups of the modular group $\SL_2(\Z).$ It is evident that modularity is the work horse in those methods. However, one could ask if the Cirlce Method can be used to find exact formulas for the Fourier coefficients of objects that still have some automorphic structure. Natural candidates are {\it mock theta functions}. Important examples are Ramanujan's third order mock theta functions (see \cite{BringmannFolsomOnoRolen})

\begin{align*}
f(q) & \coloneqq \sum_{n \geq 0} \frac{q^{n^2}}{(-q;q)_n^2}, \\
\varphi(q) &\coloneqq \sum_{n \geq 0} \frac{q^{n^2}}{(-q^2;q^2)_n},\\
\omega(q) &\coloneqq \sum_{n \geq 0} \frac{q^{2n(n+1)}}{(q;q^2)_{n+1}^2}.
\end{align*}

These are not modular forms, but by groundbreaking work of Zwegers \cite{Zwegers} can be understood as holomorphic parts of harmonic Maas forms of weight $\frac{1}{2}$ \cite{BringmannFolsomOnoRolen}. Bringmann and Ono \cite{BringmannOno} have found exact formulas for the coefficients of harmonic Maas forms of weight $< \frac{1}{2}$. This includes all the above examples but their technique does not use the Cirlce Method and rather relies on Poincaré series. Another class of automorphic objects which is not covered by the above cases is the space of modular forms tensored with harmonic Maass forms. Following the definition in \cite{BringmannFolsomOnoRolen} we are interested in the holomorphic parts of these objects, which are known as mixed-mock modular forms (see \cite{BringmannFolsomOnoRolen} for more details). It is unlikely that a basis of Poincaré series exists for this space, so the Circle Method is the method of choice to study the asymptotic behaviour of the coeffcients of these objects. In \cite{BringmannMahlburg} Bringmann and Mahlburg worked out a technique based on the classical Hardy-Ramanujan Circle Method to obtain asymptotics for the Fourier coeffcients and was later refined by Bridges and Bringmann \cite{BridgesBringmann} to provide exact formulas of Rademacher-type. They illustrated the method for the function $p_2(n)$ which counts partitions without sequences in its parts and it is the first example of an exact formula for the Fourier coefficients of a mixed-mock modular form of weight $0.$ We follow their developed method \cite{BridgesBringmann, BringmannMahlburg} closely to study {\it lower $1$-run overpartitions}. Introduced by Bringmann, Holroyd, Mahlburg, and Vlasenko \cite{BringmannHolroydMahlburgVlasenko} lower $1$-run overpartitions are defined as overpartitons in which any overlined part must occur within a run of exactly $k$ consecutive overlined parts that end below with a gap, where an overpartition is said to have a {\it gap} at $m$ if there are no parts of size $m.$ We will denote following \cite{BringmannHolroydMahlburgVlasenko} the lower $1$-run overpartitions of $n$ by $\overline{p_1}(n).$

\begin{example}
	The lower $1$-run overpartitions of size $4$ are
	
	$$\overline{4},\quad \overline{3} + \overline{1},\quad \overline{3}+1, \quad 3 + \overline{1},\quad 2 + \overline{2},\quad 2 + 1 + \overline{1},\quad 1 + 1 + 1 + \overline{1},$$
	
	\noindent
	together with the $5$ partitions of $4,$ thus $\overline{p_1}(4) = 7 + 5 = 12.$
	
\end{example}

The generating function for $\overline{p_1}(n)$ was worked by Bringmann, Holroyd, Mahlburg, and Vlasenko \cite{BringmannHolroydMahlburgVlasenko} (note that there is a typo in their statement)

$$G_1(q) \coloneqq  \frac{(q^4;q^4)_{\infty}}{(q;q)_{\infty}(q^2;q^2)_{\infty}} \varphi(q).$$

\noindent
Using Ingham's Tauberian Theorem \cite{Ingham} they showed the following asymptotic formula \cite{BringmannHolroydMahlburgVlasenko}

\begin{equation}\label{Asymptotics}
\overline{p_1}(n) \sim \frac{\sqrt{5}}{4\sqrt{6}n}e^{\pi\sqrt{\frac{5}{6}n}}, \quad (n\rightarrow \infty).
\end{equation}

\noindent
The generating function $G_1(q)$ is a mixed-mock modular form of overall weight $0$ and thus falls into the preivously mentioned framework developed by Bringmann and Bridges \cite{BridgesBringmann}. We will follow the approach of \cite{BridgesBringmann, BringmannMahlburg} closely to obtain an following exact formula for $\overline{p_1}(n)$. Many of the computations in the proof will be completely analogous to \cite{BridgesBringmann}. By using a similar structure to \cite{BridgesBringmann} we designed this article in such a way that the proofs for the exact formulas of Rademacher-type for $p_2(n)$ (in \cite{BridgesBringmann}) and $\overline{p_1}(n)$ can be easily compared. Our main result reads as follows.

\begin{theorem}\label{maintheorem}
	We have for $n \geq 1,$
	
	\begin{align*}
	\overline{p_1}(n) &= \frac{\pi}{12\sqrt{6n}} \sum_{\substack{k\geq 1 \\ \gcd(4,k)=1}} \frac{1}{k^2} \sum_{\nu \pmod{k}} (-1)^{n+\nu} K_k^{[12]}(\nu,n) \mathcal{I}_{\frac{1}{24},k,\nu}(n)\\
	&+ \frac{5\pi}{12\sqrt{6n}} \sum_{\substack{k\geq 1 \\ \gcd(4,k)=2}} \frac{1}{k^2} \sum_{\nu \pmod{k}} (-1)^{n+\nu} K_k^{[22]}(\nu,n) \mathcal{I}_{\frac{5}{12},k,\nu}(n).
	\end{align*}
\end{theorem}

\noindent
The sums $K_k^{[12]}(\nu,n)$ and $K_k^{[22]}(\nu,n)$ are defined in Section 3. Furthermore, we have set here for $b \in \R, k \in \N,$ and $\nu \in \Z$ 

$$\mathcal{I}_{b,k,\nu}(n) := \int_{-1}^{1} \frac{\sqrt{1-x^2}I_1\left(\frac{2\pi}{k}\sqrt{2bn(1-x^2)}\right)}{\cosh\left(\frac{\pi i}{k}\left(\nu - \frac{1}{6}\right)-\frac{\pi}{k}\sqrt{\frac{b}{3}}x\right)} dx.$$ 

\noindent
These integrals should be viewed as a natural result of mixing the Mordell integrals and modular factors in the Circle Method. The Mordell integrals appear in the transformation laws of mock theta functions \cite{BringmannFolsomOnoRolen} and Bessel functions do appear in the exact formulas for eta-quotients \cite{Zuckerman}. In that sense $\mathcal{I}_{b,k,\nu}$ is a natural result of combining Mordell integrals and Bessel functions. As an immediate corollary, we will give another proof of \eqref{Asymptotics}. Our strategy is as follows. In Section $2$ we will work out all modular transformation laws needed for the Circle Method. It is worthwhile to note that we will not work with $G_1(q)$ directly. Instead we are going to twist the generating function by a root of unity and work with 

\begin{equation}\label{GeneratingFunction}
\overline{G_1}(q) \coloneqq \sum_{n=0}^{\infty} (-1)^n\overline{p_1}(n) q^n = \frac{(q^4;q^4)_{\infty}}{(-q;-q)_{\infty}(q^2;q^2)_{\infty}} \varphi(-q).
\end{equation}

\noindent
The main reason for this choice is that the modular properties of \eqref{GeneratingFunction} are easier to work with, than for the other generating function, see Section $2.$ In Section $3$ we are going to give bounds for the Kloosterman sums, and the Mordell integrals that appear in the transformations laws for mock theta functions. Finally, in Section $4$ we are going to apply the Cirlce Method to prove Theorem \ref{maintheorem}.

\section*{Acknowledgements}
The author wishes to thank Kathrin Bringmann for suggesting this problem and helpful discussions. Furthermore, the author thanks Giulia Cesana for advice on rewriting Kloosterman sums and aide to verify the main result numerically. Finally, the author thanks the referee for his careful review, finding several typos in earlier versions of this article, and plenty of suggestions for improvement. The author recieved funding from the European Research Council (ERC) under the European Union’s Horizon 2020 research and innovation programme (grant agreement No. 101001179).

\section{Modularity Properties}

As stated in the introduction we define 
$$P(q):= (q;q)_{\infty}^{-1}.$$

\noindent
For coprime integers $h$ and $k,$ we find $h'$ such that $hh' \equiv -1 \pmod{k}.$ We will assume that whenever $3 \mid k$ the congruence holds $\pmod{3k}$. Aditionally, for all even $k$ we require that the  congruence holds as well $\pmod{16k}$. Furthermore, we define by $z$ a complex variable that satisfies $\re{z}>0$ and $q=e^{\frac{2\pi i}{k}(h+iz)},$ and set $q_1 := e^{\frac{2\pi i}{k}(h'+iz^{-1})}.$ The classical modular transformation law is \cite{Andrews}

\begin{equation}\label{ModularTransformationLaw}
P(q) = \omega_{h,k}z^{\frac{1}{2}}e^{\frac{\pi(z^{-1}-z)}{12k}}P(q_1),
\end{equation}

\noindent
where $\omega_{h,k}$ is the $24k$-th root of unity defined by 
$$\omega_{h,k} \coloneqq \begin{cases}
\left(\frac{-k}{h}\right) e^{-\pi i\left(\frac{1}{4}(2-hk-h) + \frac{1}{12}\left(k-\frac{1}{k}\right)\left(2h-h'+h^2h'\right)\right)} & \text{ if } h \text{ is odd,}  \\

\left(\frac{-h}{k}\right) e^{-\pi i\left(\frac{1}{4}(k-1) + \frac{1}{12}\left(k-\frac{1}{k}\right)\left(2h-h'+h^2h'\right)\right)} & \text{ if } k \text{ is odd,} 

\end{cases}$$

\noindent
and $\left(\frac{\cdot}{\cdot}\right)$ denotes the Kronecker symbol. From \eqref{ModularTransformationLaw} we obtain easily a transformation law for $P(q^r),$ where $r$ is a positive integer. To state the transformation law in a convenient form, we introduce some notation following \cite{BringmannMahlburg}. First, we set $g_r := \gcd(r,k),$ and define the quantities $\rho_r := \frac{r}{g_r}$ and $k_r := \frac{k}{g_r}.$ Note that $\gcd(\rho_r,k_r)=1.$ Finally, we have (see for instace \cite{BringmannMahlburg}) the transformation law

\begin{equation}\label{ModularTransformationLaw2}
P(q^r) = \omega_{h\rho_r,k_r}(\rho_r z)^{\frac{1}{2}} e^{\frac{\pi}{12k_r}\left(\frac{z^{-1}}{\rho_r}-\rho_r z\right)}P\left(\zeta q_1^{\frac{g_r}{\rho_r}}\right).
\end{equation}
where $\zeta \coloneqq \zeta(r, h, k)$ is a certain root of unity whose exact value is irrelevant for our purposes. Therefore, we will assume without loss of generality that $\zeta = 1$.

We are now going to study the necessary transformation laws to sudy \eqref{GeneratingFunction}. We exploit the following linear relation, discovered by Ramanujan \cite{RamanujanCollectedPapers} (see \cite{Watson} for a proof), between the two third-order mock theta functions $\varphi(q)$ and $f(q):$

$$2\varphi(-q) - f(q) = \frac{(q;q)_{\infty}^2}{(-q;q)_{\infty}(q^2;q^2)_{\infty}}.$$

\noindent
We can then rewrite the generating function \eqref{GeneratingFunction} as

$$\overline{G_1}(q) = g_1(q) + g_2(q),$$
where

\begin{align*}
g_1(q) &\coloneqq \frac{(q^4;q^4)_{\infty}^2(q;q)_{\infty}}{2(q^2;q^2)_{\infty}^4}f(q), \\
g_2(q) &\coloneqq \frac{(q^4;q^4)_{\infty}^2 (q;q)_{\infty}^4}{2(q^2;q^2)_{\infty}^6}.
\end{align*}

We start with the transformation law for $g_1(q).$ Define $\xi(q):= \frac{P(q^2)^4}{P(q^4)^2P(q)}.$ We will have to distinguish three cases, depending on the different values of $\gcd(4,k).$ These transformations all follow from iterated application of \eqref{ModularTransformationLaw} and \eqref{ModularTransformationLaw2}.

First, assume that $4|k.$ Then,

$$\xi(q) = \frac{\omega_{h,\frac{k}{2}}^4}{\omega_{h,k}\omega_{h,\frac{k}{4}}^2} z^{\frac{1}{2}} e^{-\frac{\pi}{12k}(z^{-1}-z)}\xi(q_1).$$

Next, if $\gcd(4,k) = 2,$ we have

$$\xi(q) = \frac{\omega_{h,\frac{k}{2}}^4}{2 \omega_{h,k} \omega_{2h,\frac{k}{2}}^2}  z^{\frac{1}{2}} e^{\frac{5\pi}{12kz}+\frac{\pi z}{12k}} \frac{P(q_1^2)^4}{P(q_1)^3}.$$

Finally, if $\gcd(4,k)=1,$

$$\xi(q) = \frac{\omega_{2h,k}^4}{\omega_{h,k}\omega_{4h,k}^2} z^{\frac{1}{2}} e^{\frac{\pi}{24kz} + \frac{\pi z}{12k}} \frac{P\left(q_1^{\frac{1}{2}}\right)^4}{P(q_1)P\left(q_1^{\frac{1}{4}}\right)^2}.$$

Now we turn to the third order mock theta function $f(q).$ Andrews \cite{Andrews1} has shown that $f(q)$ behaves essentially modular of level $2.$ If $k$ is even we have

\begin{align*}
f(q) &= (-1)^{\frac{k}{2}+1}e^{\pi i\left(\frac{h'}{2}-\frac{3h'k}{4}\right)} \omega_{h,k} z^{-\frac{1}{2}} e^{\frac{\pi(z^{-1}-z)}{12k}}f(q_1) \nonumber \\
&+ \omega_{h,k} \frac{2}{k} z^{\frac{1}{2}} e^{-\frac{\pi z}{12k}} \sum_{\nu \pmod{k}} (-1)^{\nu} e^{\frac{\pi i h' (-3\nu^2+\nu)}{k}} I_{k,v}(z),
\end{align*}

\noindent
where $I_{k,\nu}(z)$ is the Mordell-integral defined by 

$$I_{k,\nu}(z) \coloneqq \int_{-\infty}^{\infty} \frac{e^{-\frac{3\pi z x^2}{k}}}{\cosh\left(\frac{\pi i \left(\nu - \frac{1}{6}\right)}{k} - \frac{\pi z x}{k}\right)} dx.$$

\noindent
Note that, as mentioned by Bringmann and Mahlburg in \cite{BringmannMahlburg}, there is a typo regarding the term $e^{\frac{\pi(z^{-1}-z)}{12k}}$ as it is stated in the statement of Theorem $2.2$ in \cite{Andrews1}. We also stated Andrews formula in the same way as Bringmann and Mahlburg did in \cite{BringmannMahlburg} by replacing the $2k$ in Theorem $2.2$ in \cite{Andrews1} by an even $k$. If $k$ is odd we have

\begin{align*}
f(q) &= 2(-1)^{\frac{1}{2}(k-1)} e^{\frac{3\pi i h'}{4k}} \omega_{h,k} z^{-\frac{1}{2}} e^{-\frac{2\pi}{3kz}-\frac{\pi z}{12k}} \omega\left(q_1^{\frac{1}{2}}\right) \nonumber \\
&+ \frac{2}{k} \omega_{h,k} z^{\frac{1}{2}} e^{-\frac{\pi z}{12k}} \sum_{\nu \pmod{k}} (-1)^{\nu} e^{\frac{\pi i h' (-3\nu^2-\nu)}{k}}I_{k,\nu}(z),
\end{align*}

\noindent
where $\omega(q)$ is the third order mock theta function defined in Section 1.

Finally, we are going to derive the transformation laws for $g_2(q)$ depending on $\gcd(4,k).$ If $4|k$ we have 

$$g_2(q) = \frac{\omega_{h,\frac{k}{2}}^6}{\omega_{h,k}^4\omega_{h,\frac{k}{4}}^2} g_2(q_1).$$

\noindent
Next, if $\gcd(4,k)=2$ we have 

$$ g_2(q) = \frac{\omega_{h,\frac{k}{2}}^6}{4\omega_{h,k}^4\omega_{2h,\frac{k}{2}}^2} e^{\frac{\pi}{2kz}} \frac{P(q_1^2)^6}{P(q_1)^6}.$$

\noindent
Finally, if $\gcd(4,k)=1$ we have

$$ g_2(q) = \frac{\omega_{2h,k}^6}{\omega_{h,k}^4\omega_{4h,k}^2} e^{-\frac{\pi}{8kz}} \frac{P\left(q_1^{\frac{1}{2}}\right)^6}{P(q_1)^4P\left(q_1^{\frac{1}{4}}\right)^2}.$$

\section{Bounds on Kloosterman sums and Mordell Integrals}
As in \cite{BridgesBringmann} it will be important for use to bound certain Kloosterman sums appearing in the Circle Method. Analogous to \cite{BridgesBringmann} we are going to rewrite the Kloosterman sums in question into classical Kloosterman sums and use Lemma \ref{RademacherKloostermanSumBound} stated below to bound them. In the next section we define $k_1$ as the denominator of the fraction proceeding $\frac{h}{k}$ in the Farey sequence of order $N \in \N.$ We use the following bounds for the Kloosterman sums.

\begin{lemma}\label{RademacherKloostermanSumBound}
	We have for $k\in  \N, n, m, \ell \in \Z, n \neq 0$ with $N +1 \leq \ell \leq N + k + 1,$ for $\varepsilon > 0,$
	
	\begin{align*}
	K_k(n,m) &\coloneqq \sum_{\substack{h \pmod{k} \\ \gcd(h,k)=1}} e^{-\frac{2\pi i}{k}(nh-mh')} = O_{\varepsilon}\left(k^{\frac{2}{3}+\varepsilon}\gcd(|n|,k)^{\frac{1}{3}}\right), \\
	\mathbb{K}_{k,\ell}(n,m) &\coloneqq \sum_{\substack{h \pmod{k} \\ \gcd(h,k)=1 \\ N < k+k_1 \leq \ell}} e^{-\frac{2\pi i}{k}(nh-mh')} = O_{\varepsilon}\left(k^{\frac{2}{3}+\varepsilon}\gcd(|n|,k)^{\frac{1}{3}}\right).
	\end{align*}
\end{lemma}
	\begin{proof}
	The individual statements of the Lemma can be found in Rademacher's article \cite{RademacherJFunction}. The first estimate for the classical Kloosterman sums $K_k(n,m)$ is as described in \cite{RademacherJFunction} the result of a series of contributions by Kloosterman \cite{Kloosterman1, Kloosterman2}, Estermann \cite{Estermann}, Salié \cite{Salie}, and Davenport \cite{Davenport}. The estimate appears in the form we require it (and used by Rademacher in \cite{RademacherJFunction}) in Salié's thesis \cite{Salie}. The article by Davenport \cite{Davenport} in independent work shows for prime modulus $p$ the bound
	\[
		|K_p(n,m)| < 2p^{\frac{1}{3}}.
	\]
	
	The claimed bound for the incomplete Kloosterman sums is stated first in the article \cite{RademacherJFunction} of Rademacher. Rademacher proved it in the same article in a footnote combining results of Estermann \cite{Estermann}, Salié \cite{Salie}, and Davenport \cite{Davenport}.
	\end{proof}

	We use throughout this section, following \cite{BridgesBringmann}, the notation $[a]_k$ to denote the inverse of $a \pmod{k}.$

	In the case that $4|k$ define the following Kloosterman sums
	
	\begin{align*}
	K_k^{[41]}(n,m) &\coloneqq \sum_{\substack{h \pmod{k} \\ \gcd(h,k)=1}} \frac{\omega_{h,\frac{k}{2}}^4}{\omega_{h,\frac{k}{4}}^2} e^{\frac{\pi i h'}{2}\left(1-\frac{3k}{2}\right)} e^{\frac{2\pi i}{k}(-nh+mh')},\\
	K_k^{[42]}(\nu,n,m) &\coloneqq \sum_{\substack{h \pmod{k} \\ \gcd(h,k)=1}} \frac{\omega_{h,\frac{k}{2}}^4}{\omega_{h,\frac{k}{4}}^2} e^{\frac{\pi ih'}{k}(-3\nu^2+\nu)}e^{\frac{2\pi i}{k}(-nh+mh')},\\
	K_k^{[43]}(n,m) &\coloneqq \sum_{\substack{h \pmod{k} \\ \gcd(h,k)=1}} \frac{\omega_{h,\frac{k}{2}}^6}{\omega_{h,k}^4\omega_{h,\frac{k}{4}}^2} e^{\frac{2\pi i}{k}(-nh+mh')}.
	\end{align*}
	
	Furthermore, we define for $N+1 \leq \ell \leq N + k -1,$
	
	\begin{align*}
	\mathbb{K}_{k,\ell}^{[41]}(n,m) &\coloneqq \sum_{\substack{0\leq h < k \\ \gcd(h,k)=1 \\ N < k+k_1 \leq \ell}} \frac{\omega_{h,\frac{k}{2}}^4}{\omega_{h,\frac{k}{4}}^2} e^{\frac{\pi i h'}{2}\left(1-\frac{3k}{2}\right)} e^{\frac{2\pi i}{k}(-nh+mh')},\\
	\mathbb{K}_{k,\ell}^{[42]}(\nu,n,m) &\coloneqq \sum_{\substack{0\leq h < k \\ \gcd(h,k)=1 \\ N < k+k_1 \leq \ell}} \frac{\omega_{h,\frac{k}{2}}^4}{\omega_{h,\frac{k}{4}}^2} e^{\frac{\pi ih'}{k}(-3\nu^2+\nu)}e^{\frac{2\pi i}{k}(-nh+mh')},\\
	\mathbb{K}_{k,\ell}^{[43]}(n,m) &\coloneqq \sum_{\substack{0\leq h < k \\ \gcd(h,k)=1 \\ N < k+k_1 \leq \ell}} \frac{\omega_{h,\frac{k}{2}}^6}{\omega_{h,k}^4\omega_{h,\frac{k}{4}}^2} e^{\frac{2\pi i}{k}(-nh+mh')}.
	\end{align*}
	
	\begin{lemma}\label{BoundsOnKloostermanSumsCase4|K}
		We have for $\varepsilon > 0,$ 
		
		$$K_k^{[41]}(n,m),K_k^{[42]}(\nu,n,m),K_k^{[43]}(n,m),\mathbb{K}_{k,\ell}^{[41]}(n,m),\mathbb{K}_{k,\ell}^{[42]}(\nu,n,m),\mathbb{K}_{k,\ell}^{[43]}(n,m) \ll_{\varepsilon} n^{\frac{1}{3}}k^{\frac{2}{3}+\varepsilon}.$$
	\end{lemma}

	\begin{proof}
		The idea is to rewrite the multipliers in the above Kloosterman sums in such a way that we obtain a classical Kloosterman sum, which we bound from above with Lemma \ref{RademacherKloostermanSumBound}. We remark briefly that $K_k^{[41]}(m,n), K_k^{[42]}(\nu,m,n), \mathbb{K}_k^{[41]}(m,n),$ and $\mathbb{K}_k^{[42]}(\nu,m,n)$ do have up to a simple term the same multiplier. We note that
		\[
			e^{\frac{\pi i h'}{k}(-3\nu^2 +\nu)} = e^{\frac{2\pi i h'}{k}\left(\frac{-3\nu^2+\nu}{2}\right)},
		\]
		where $\frac{-3\nu^2+\nu}{2}$ is an integer for all $\nu \pmod{k}$. Thus, it will be enough to consider $K_k^{[41]}(m,n)$ and $K_k^{[43]}(m,n)$. We have for the multiplier of $K_k^{[41]}$ keeping in mind that $hh' \equiv -1 \pmod{16k}$
		
		$$\frac{\omega_{h,\frac{k}{2}}^4}{\omega_{h,\frac{k}{4}}^2} = -e^{-\frac{\pi i}{8}(-h'k +h^2h'k-h(4+k))} = - e^{\frac{\pi i}{8}(h'k +(4+2k)h)}.$$

		Hence,
		
		\begin{align*}
		K_k^{[41]}(n,m) &\coloneqq \sum_{\substack{h \pmod{k} \\ \gcd(h,k)=1}} \frac{\omega_{h,\frac{k}{2}}^4}{\omega_{h,\frac{k}{4}}^2} e^{\frac{\pi i h'}{2}\left(1-\frac{3k}{2}\right)} e^{\frac{2\pi i}{k}(-nh+mh')} \\
		&= -\sum_{\substack{h \pmod{k} \\ \gcd(h,k)=1}} e^{-\frac{2\pi i}{k}\left(\left(n-\frac{2k^2+4k}{16}\right)h-\left(m-\frac{5k^2-4k}{16}\right)h'\right)}.
		\end{align*}
		
		Since $4|k$ we have 
		
		$$\frac{2k^2+4k}{16} \in \Z, \quad \frac{5k^2-4k}{16} \in \Z.$$
		
		Therefore, we can use Lemma \ref{RademacherKloostermanSumBound} to obtain as desired
		
		$$\left|K_k^{[41]}(n,m)\right| = \left|K_k\left(n-\frac{2k^2+4k}{16},m-\frac{5k^2-4k}{16}\right)\right| \ll_{\varepsilon} 
		k^{\frac{2}{3}+\varepsilon} \gcd\left(\left|n-\frac{2k^2+4k}{16}\right|,k\right)^{\frac{1}{3}} \ll_{\varepsilon} n^{\frac{1}{3}}k^{\frac{2}{3}+\varepsilon}.$$
		
		For the Kloosterman sum $K_k^{[43]}$ we rewrite the multiplier as
		
		$$\frac{\omega_{h,\frac{k}{2}}^6}{\omega_{h,k}^4\omega_{h,\frac{k}{4}}^2} = e^{-\frac{k\pi i}{8}(2h+h')}=e^{-\frac{2\pi i}{k}\left(\frac{k^2h}{8}+\frac{k^2h'}{16}\right)}.$$
		
		Thus, since $4 | k$ we have 
		
		$$\frac{k^2}{8} \in \Z, \quad \frac{k^2}{16} \in \Z.$$
		
		Hence, we can apply Lemma \ref{RademacherKloostermanSumBound} to achieve the desired bound
		
		$$\left|K_k^{[43]}(n,m)\right| = \left|K_k\left(n+\frac{k^2}{8},m-\frac{k^2}{16}\right)\right| \ll_{\varepsilon} 
		k^{\frac{2}{3}+\varepsilon} \gcd\left(\left|n-\frac{k^2}{8}\right|,k\right)^{\frac{1}{3}} \ll_{\varepsilon} n^{\frac{1}{3}}k^{\frac{2}{3}+\varepsilon}.$$
		
	\end{proof}

In the case that $\gcd(4,k)=2$ define the following Kloosterman sums

\begin{align*}
K_k^{[21]}(n,m) &\coloneqq \sum_{\substack{h \pmod{k} \\ \gcd(h,k)=1}} \frac{\omega_{h,\frac{k}{2}}^4}{\omega_{2h,\frac{k}{2}}^2} e^{\frac{\pi i h'}{2}\left(1-\frac{3k}{2}\right)} e^{\frac{2\pi i}{k}(-nh+mh')},\\
K_k^{[22]}(\nu,n,m) &\coloneqq \sum_{\substack{h \pmod{k} \\ \gcd(h,k)=1}} \frac{\omega_{h,\frac{k}{2}}^4}{\omega_{2h,\frac{k}{2}}^2} e^{\frac{\pi ih'}{k}(-3\nu^2+\nu)}e^{\frac{2\pi i}{k}(-nh+mh')},\\
K_k^{[23]}(n,m) &\coloneqq \sum_{\substack{0\leq h < k \leq N \\ \gcd(h,k)=1}} \frac{\omega_{h,\frac{k}{2}}^6}{\omega_{h,k}^4\omega_{2h,\frac{k}{2}}^2} e^{\frac{2\pi i}{k}(-nh+mh')}.
\end{align*}

Furthermore, we need the incomplete Kloosterman sums for $N+1 \leq \ell \leq N + k -1,$ define

\begin{align*}
\mathbb{K}_{k,\ell}^{[21]}(n,m) &\coloneqq \sum_{\substack{0\leq h < k \\ \gcd(h,k)=1 \\ N < k+k_1 \leq \ell}} \frac{\omega_{h,\frac{k}{2}}^4}{\omega_{2h,\frac{k}{2}}^2} e^{\frac{\pi i h'}{2}\left(1-\frac{3k}{2}\right)} e^{\frac{2\pi i}{k}(-nh+mh')},\\
\mathbb{K}_{k,\ell}^{[22]}(\nu,n,m) &\coloneqq \sum_{\substack{0\leq h < k \\ \gcd(h,k)=1 \\ N < k+k_1 \leq \ell}} \frac{\omega_{h,\frac{k}{2}}^4}{\omega_{2h,\frac{k}{2}}^2} e^{\frac{\pi ih'}{k}(-3\nu^2+\nu)}e^{\frac{2\pi i}{k}(-nh+mh')},\\
\mathbb{K}_{k,\ell}^{[23]}(n,m) &\coloneqq \sum_{\substack{0\leq h < k \\ \gcd(h,k)=1 \\ N < k+k_1 \leq \ell}} \frac{\omega_{h,\frac{k}{2}}^6}{\omega_{h,k}^4\omega_{2h,\frac{k}{2}}^2} e^{\frac{2\pi i}{k}(-nh+mh')}.
\end{align*}

\begin{lemma}\label{BoundsOnKloostermanSumsCaseGCD(4,k)=2}
		We have for $\varepsilon > 0,$ 
		
		$$K_k^{[21]}(n,m),K_k^{[22]}(\nu,n,m),K_k^{[23]}(n,m),\mathbb{K}_{k,\ell}^{[21]}(n,m),\mathbb{K}_{k,\ell}^{[22]}(\nu,n,m),\mathbb{K}_{k,\ell}^{[23]}(n,m) \ll_{\varepsilon} n^{\frac{1}{3}}k^{\frac{2}{3}+\varepsilon}.$$
\end{lemma}

\begin{proof}
	We will later show  $K_k^{[21]}(n,m) = -K_k^{[23]}(n,m).$ Thus, is suffices to consider $K_k^{[21]}(n,m).$ We first consider the case $3\mid k.$ When evaluating $\omega_{2h,\frac{k}{2}}$ we will have to take care as this requires a choice of $\tilde{h'}$ such that $2h\tilde{h'} \equiv -1 \pmod{\frac{k}{2}}.$ Starting from a general $h'$ which satisfies $hh' \equiv -1 \pmod{48k}$ we define $\tilde{h'} \coloneqq [2]_{\frac{k}{2}}h'$. Indeed, we have as required
	\[
		2h\tilde{h'} = 2[2]_{\frac{k}{2}}hh' \equiv -1 \pmod{\frac{k}{2}}.
	\]
	With this choice we evaluate the multiplier
	\[
		\frac{\omega_{h,\frac{k}{2}}^4}{\omega_{2h,\frac{k}{2}}^2} = e^{\frac{2\pi i}{24k}\left(3k^2-30k + \left(-16[2]_{\frac{k}{2}}hh'+8hh'+12k+6k^2-2hh'k^2+4[2]_{\frac{k}{2}}hh'k^2\right)h+\left(4[2]_{\frac{k}{2}}-8-[2]_{\frac{k}{2}}k^2+2k^2\right)h'\right)}
	\]
	Since we assumed that $3 \mid k$ we can rewrite this as
	\[
		e^{\frac{2\pi i}{24k}\left(3k^2-30k + \left(16[2]_{\frac{k}{2}}-8+12k+8k^2-4[2]_{\frac{k}{2}}k^2\right)h+\left(4[2]_{\frac{k}{2}}-8-[2]_{\frac{k}{2}}k^2+2k^2\right)h'\right)}.
	\]
	Now note that since $6\mid k$ we have
	\[
		\frac{16[2]_{\frac{k}{2}}-8+12k+8k^2-4[2]_{\frac{k}{2}}k^2}{24} \in \Z, \quad \frac{4[2]_{\frac{k}{2}}-8-[2]_{\frac{k}{2}}k^2+2k^2}{24} \in \Z.
	\]
	For the case $3 \nmid k$ we follow a similar approach and consider $\tilde{h'} =3[3]_{k}h'$. It follows that in this case
	\[
		\frac{\omega_{h,\frac{k}{2}}^4}{\omega_{2h,\frac{k}{2}}^2} = e^{\frac{2\pi i}{24k}\left(3k^2-30k + \left(48[2]_{\frac{k}{2}}[3]_k-24[3]_k+12k+6k^2+6[3]_kk^2-12[2]_{\frac{k}{2}}[3]_kk^2\right)h+3[3]_k\left(4[2]_{\frac{k}{2}}-8-[2]_{\frac{k}{2}}k^2+2k^2\right)h'\right)}.
	\]
	One obtains as before
	\[
	\frac{48[2]_{\frac{k}{2}}[3]_k-24[3]_k+12k+6k^2+6[3]_kk^2-12[2]_{\frac{k}{2}}[3]_kk^2}{24} \in \Z, \quad \frac{3[3]_k(4[2]_{\frac{k}{2}}-8-[2]_{\frac{k}{2}}k^2+2k^2)}{24} \in \Z.
	\]
	
	Hence, in both cases the multiplier can be rewritten such that it is in the shape of the classical Kloosterman sum required to apply Lemma \ref{RademacherKloostermanSumBound}. We are left to show that the last factor can be rewritten as well in such a shape. We write
	\[
		e^{\frac{\pi i h'}{2}\left(1-\frac{3k}{2}\right)} = e^{-\frac{2\pi i h'}{k}\left(\frac{3k^2-2k}{8}\right)}
	\]
	where we find readily that $\frac{3k^2-2k}{8}$ is an integer, as desired.
\end{proof}

In the case that $\gcd(4,k)=1$ define the following Kloosterman sums

\begin{align*}
K_k^{[11]}(n,m) &\coloneqq \sum_{\substack{h \pmod{k} \\ \gcd(h,k)=1}} \frac{\omega_{2h,k}^4}{\omega_{4h,k}^2}
 e^{\frac{3\pi ih'}{4k}} e^{\frac{2\pi i}{k}(-nh+\frac{mh'}{4})},\\
K_k^{[12]}(\nu,n,m) &\coloneqq \sum_{\substack{h \pmod{k} \\ \gcd(h,k)=1}} \frac{\omega_{2h,k}^4}{\omega_{4h,k}^2} e^{\frac{\pi ih'}{k}(-3\nu^2-\nu)}e^{\frac{2\pi i}{k}(-nh+\frac{mh'}{4}')},\\
K_k^{[13]}(n,m) &\coloneqq \sum_{\substack{0\leq h < k \leq N \\ \gcd(h,k)=1}} \frac{\omega_{2h,k}^6}{\omega_{h,k}^4\omega_{4h,k}^2} e^{\frac{2\pi i}{k}(-nh+\frac{mh'}{4})}.
\end{align*}

Furthermore, we need the incomplete Kloosterman sums for $N+1 \leq \ell \leq N + k -1,$ define

\begin{align*}
\mathbb{K}_{k,\ell}^{[11]}(n,m) &\coloneqq \sum_{\substack{0\leq h < k \\ \gcd(h,k)=1 \\ N < k+k_1 \leq \ell}} \frac{\omega_{2h,k}^4}{\omega_{4h,k}^2} e^{\frac{3\pi ih'}{4k}} e^{\frac{2\pi i}{k}(-nh+\frac{mh'}{4})},\\
\mathbb{K}_{k,\ell}^{[12]}(\nu,n,m) &\coloneqq \sum_{\substack{0\leq h < k \\ \gcd(h,k)=1 \\ N < k+k_1 \leq \ell}} \frac{\omega_{2h,k}^4}{\omega_{4h,k}^2} e^{\frac{\pi ih'}{k}(-3\nu^2-\nu)}e^{\frac{2\pi i}{k}(-nh+\frac{mh'}{4})},\\
\mathbb{K}_{k,\ell}^{[13]}(n,m) &\coloneqq \sum_{\substack{0\leq h < k \\ \gcd(h,k)=1 \\ N < k+k_1 \leq \ell}} \frac{\omega_{2h,k}^6}{\omega_{h,k}^4\omega_{4h,k}^2} e^{\frac{2\pi i}{k}(-nh+\frac{mh'}{4})}.
\end{align*}

\begin{lemma}\label{BoundsOnKloostermanSumsCaseGCD(4,k)=1}
	We have for $\varepsilon > 0,$ 
	
	$$K_k^{[11]}(n,m),K_k^{[12]}(\nu,n,m),K_k^{[13]}(n,m),\mathbb{K}_{k,\ell}^{[11]}(n,m),\mathbb{K}_{k,\ell}^{[12]}(\nu,n,m),\mathbb{K}_{k,\ell}^{[13]}(n,m) \ll_{\varepsilon} n^{\frac{1}{3}}k^{\frac{2}{3}+\varepsilon}.$$
\end{lemma}

\begin{proof}
	We first observe that we can rewrite some of the common terms as follows
	\[
		e^{\frac{3\pi ih'}{4k}}=e^{\frac{6\pi i[8]_kh'}{k}}, \quad e^{\frac{2\pi i}{k}\left(-nh+\frac{mh'}{4}\right)}= e^{\frac{2\pi i}{k}(-nh+2[8]_kmh')}
	\]
	since $k$ is odd. Therefore, it is enough to rewrite the multipliers suitably. We start with the sum $K_k^{[11]}$. We will again consider the dichotomy whether $3 \mid k$ or $3 \nmid k$. Suppose first that $3 \mid k$. Thus, let $h'$ be such that $hh' \equiv -1 \pmod{3k}$. We choose for evaulating $w_{2h,k}$ and $w_{4h,k}$ the quantities $[2]_kh'$ and $[4]_kh'$, respectively as $\tilde{h'}.$ We can then write
	\[
		\frac{\omega_{2h,k}^4}{\omega_{4h,k}^2} = e^{\frac{2\pi i}{12k}\left(3k(1-k) +(1-k^2)(-8[2]_k+16[4]_k)h+(1-k^2)([4]_k-2[2]_k)h'\right)}.
	\]
	Furthermore a short computation using $3 \mid k$ shows
	\[
		\frac{(1-k^2)(-8[2]_k+16[4]_k)}{12} \in \Z, \quad \frac{(1-k^2)([4]_k-2[2]_k)}{12} \in \Z.
	\]
	
	We now assume that $3 \nmid k$. In this case we choose the quantities $[2]_k[12]_k12h'$ and $[4]_k[12]_k12h'$ to evaluate $\omega_{2h,k}$ and $\omega_{4h,k}$, respectively. We obtain
	\[
		\frac{\omega_{2h,k}^4}{\omega_{4h,k}^2} = e^{\frac{2\pi i}{12k}\left(3k(1-k) +(1-k^2)12[12]_k(-8[2]_k+16[4]_k)h+(1-k^2)12[12]_k([4]_k-2[2]_k)h'\right)}.
	\]
	
	Finally, we rewrite the sum $K_k^{[13]}$. Once more we will have two cases depending on the divisiblity of $k$ by $3$. Let us first assume that $3 \mid k$ and we have $h'$ such that $hh' \equiv -1 \pmod{3k}$. To evaluate $\omega_{h,k}, \omega_{2h,k}$, and $\omega_{4h,k}$ we choose $\tilde{h'}=[8]_k8h', \tilde{h'}=[2]_k[8]_k8h',$ and $\tilde{h'} = [4]_k[8]_k4h'$, respectively. We obtain using $3 \mid k$ the identity
	\[
		\frac{\omega_{2h,k}^6}{\omega_{h,k}^4\omega_{4h,k}^2} = e^{\frac{2\pi i}{12k}(1-k^2)((16[8]_k-96[2]_k[8]_k+128[4]_k[8]_k)h + (16[8]_k-24[2]_k[8]_k+8[4]_k[8]_k)h')}
	\]
	Furthermore, we see since $3 \mid k$ that
	\[
		\frac{(1-k^2)(16[8]_k-96[2]_k[8]_k+128[4]_k[8]_k)}{12} \in \Z, \quad \frac{(1-k^2)(16[8]_k-24[2]_k[8]_k+8[4]_k[8]_k)}{12} \in \Z.
	\]
	
	The final case we have to consider is $3 \nmid k$. This time we choose to evaluate $\omega_{h,k}, \omega_{2h,k}$, and $\omega_{4h,k}$ with $\tilde{h'}=[24]_k24h', \tilde{h'}=[2]_k[24]_k24h',$ and $\tilde{h'} = [4]_k[24]_k24h'$, respectively. We obtain
	\[
		\frac{\omega_{2h,k}^6}{\omega_{h,k}^4\omega_{4h,k}^2} = e^{\frac{2\pi i}{k}(1-k^2)(2[24]_k(2-12[2]_k+16[4]_k)h + 2[24]_k(2-3[2]_k+[4]_k)h')}.
	\]
	Hence, we can rewrite all the Kloosterman sums in the classical form required to apply Lemma \ref{RademacherKloostermanSumBound} which gives the desired bound.
	\end{proof}

We agree on the convention that for any Kloosterman sum above we set $K_k(\nu,n) \coloneqq K_k(\nu,n,0).$

We take the following definition from \cite{BringmannMahlburg}. For $b \in \R, k \in \N$ and $0< \nu \leq k$ define $\mathcal{J}_{b,k,\nu}(z) \coloneqq ze^{\frac{\pi b}{kz}} I_{k,\nu}(z)$ and define the principal part truncation of $\mathcal{J}_{b,k,\nu}$ by

$$ \mathcal{J}_{b,k,\nu}^{*} (z) \coloneqq \sqrt{\frac{b}{3}} \int_{-1}^{1} \frac{e^{\frac{\pi b}{kz}(1-x^2)}}{\cosh\left(\frac{\pi i}{k}\left(\nu-\frac{1}{6}\right)-\frac{\pi}{k}\sqrt{\frac{b}{3}}x\right)} dx.$$

We have the following estimate for the principal part integrals \cite{BringmannMahlburg}.

\begin{lemma}\label{BoundIntegral}
	Let $b \in \R, k \in \N$ and $0<\nu\leq k,$ then we have as $z \rightarrow 0:$
	\begin{itemize}
		
		\item[\textup{(i)}]  If $b\leq 0,$ then we have 
		
		$$ | \mathcal{J}_{b,k,\nu} (z) | \ll \frac{1}{\left|\frac{\pi}{2} - \frac{\pi}{k}\left(\nu - \frac{1}{6}\right) \right|}.$$
		
		\item[\textup{(ii)}] If $b>0,$ then $\mathcal{J}_{b,k,\nu}(z) = \mathcal{J}_{b,k,\nu}^{*}(z) + \mathcal{E}_{b,k,\nu},$ where 
		
		$$ | \mathcal{E}_{b,k,\nu} (z) | \ll \frac{1}{\left|\frac{\pi}{2} - \frac{\pi}{k}\left(\nu - \frac{1}{6}\right) \right|}.$$

	\end{itemize} 
\end{lemma}

\section{Circle Method}
In this section we are going to derive an exact formula for $\overline{p_1}(n)$ using the Circle Method. We follow the approach used in \cite{BridgesBringmann} which is based on Rademacher \cite{RademacherJFunction}. For all $n\geq 1,$ Cauchy's Theorem yields 

$$(-1)^n\overline{p_1}(n) = \frac{1}{2\pi i} \int_{C} \frac{\overline{G_1}(q)}{q^{n+1}}dq,$$

\noindent
where we choose $C$ as the circle with radius $r=e^{-\frac{2\pi}{N^2}}$ with $N \in \N$ which we take into the limit $N \rightarrow \infty.$ Futhermore, we parametrize the circle with $q = e^{-\frac{2\pi}{N^2}+2\pi i t}$ for $0 \leq t \leq 1.$ This yields

$$(-1)^n \overline{p_1}(n) = \int_{0}^{1} \overline{G_1}\left(e^{-\frac{2\pi}{N^2}+2\pi i t}\right) \cdot e^{\frac{2\pi n}{N^2}-2\pi int} dt.$$

\noindent
We are now going to decompose the circle into standard Farey arcs. Throughout we let $0 \leq h < k \leq N$ with $\gcd(h,k) = 1.$ We define 

$$ \vartheta_{h,k}' \coloneqq \frac{1}{k(k+k_1)}, \quad \vartheta_{h,k}'' \coloneqq \frac{1}{k(k+k_2)}, \quad \vartheta_{0,1}'\coloneqq \frac{1}{N+1}$$

\noindent
where $\frac{h_1}{k_1} < \frac{h}{k} < \frac{h_2}{k_2}$ are neighbouring Farey fractions in the Farey sequence of order $N.$ For notational convenience write $z =k(N^{-2}-i\Phi)$ with $\Phi = t-\frac{h}{k},$ where we have that $-\vartheta_{h,k}' \leq \Phi \leq \vartheta_{h,k}''.$ Splitting the circle along the Farey arcs then yields

\begin{align}\label{CircleMethodFormula}
(-1)^n \overline{p_1}(n) &= \sum_{\substack{0\leq h < k \leq N \\ \gcd(h,k)=1}} e^{-\frac{2\pi i h n}{k}} \int_{-\vartheta_{h,k}'}^{\vartheta_{h,k}''} \overline{G_1}\left(e^{\frac{2\pi i}{k}(h+iz)}\right)\cdot e^{\frac{2\pi n z}{k}} d\Phi \\ 
&= \sum_{\substack{0\leq h < k \leq N \\ \gcd(h,k)=1}} e^{-\frac{2\pi i h n}{k}} \int_{-\vartheta_{h,k}'}^{\vartheta_{h,k}''} \left[g_1\left(e^{\frac{2\pi i}{k}(h+iz)}\right)+g_2\left(e^{\frac{2\pi i}{k}(h+iz)}\right)\right]\cdot e^{\frac{2\pi n z}{k}} d\Phi. \nonumber
\end{align}

We have the well-known estimate (see \cite{RademacherJFunction})

\begin{equation}
\frac{1}{kN} \ll \frac{1}{k(k+k_j)} \ll \frac{1}{kN},
\end{equation}

\noindent
for $j=1,2,$ as well as the bound 

\begin{equation}\label{BoundOnInverseOfZ}
\re{z^{-1}}= \frac{N^{-2}}{kN^{-4}+k\Phi^2} \geq \frac{N^2}{k+N^2k^{-1}} \geq \frac{k}{2}.
\end{equation}

In the following we are going to split the sum in \eqref{CircleMethodFormula} according to the different values of $\gcd(4,k)$ and we will estimate these contributions individually. We write following the notation of \cite{BridgesBringmann} 

$$(-1)^n\overline{p_1}(n) = \sum_{4}+\sum_{2}+\sum_{1},$$

\noindent
where $\sum_d$ denotes the part of the sum \eqref{CircleMethodFormula} where $\gcd(4,k)=d.$ To improve our estimates we will (following again \cite{BridgesBringmann}) further decompose the integral along the Farey arc as follows

\begin{equation}\label{Splitting}
\int_{-\vartheta_{h,k}'}^{\vartheta_{h,k}''} =  \int_{-\frac{1}{k(k+N)}}^{\frac{1}{k(k+N)}} +\int_{-\frac{1}{k(k+k_1)}}^{-\frac{1}{k(k+N)}} + \int_{\frac{1}{k(k+N)}}^{\frac{1}{k(k+k_2)}}.
\end{equation}

\noindent
Furthermore, we use the refined decomposition

$$\int_{-\frac{1}{k(k+k_1)}}^{-\frac{1}{k(k+N)}} = \sum_{\ell = k +k_1}^{k+N-1} \int_{-\frac{1}{k\ell}}^{-\frac{1}{k(\ell + 1)}}.$$

\noindent
Thus, we obtain that (see \cite{BridgesBringmann})

\begin{equation}\label{Splitting2}
\sum_{\substack{0\leq h < k \leq N \\ \gcd(h,k)=1 \\ \gcd(4,k)=d}} \int_{-\frac{1}{k(k+k_1)}}^{-\frac{1}{k(k+N)}} = \sum_{\substack{1\leq k \leq N \\ \gcd(4,k)=d}} \sum_{\ell = N+1}^{k+N-1} \sum_{\substack{0\leq h < k \\ \gcd(h,k)=1 \\ N < k+k_1 \leq \ell}} \int_{-\frac{1}{k\ell}}^{-\frac{1}{k(\ell + 1)}}.
\end{equation}

We will use the obvious analogue decomposition for the last integral in \eqref{Splitting}. Before we start estimating we make three final definitions to simplify notation. We define following again \cite{BridgesBringmann} $a(n)$ and $r(n)$ by

$$ \sum_{n \geq 0} a(n)q^n = \frac{1}{2}\xi(q)f(q) = g_1(q), \quad \sum_{n \geq 0} r(n)q^n = \xi(q).$$

In what follows we will estimate the sums $\sum_d$ for $d=1,2,4$. These computations are analogous to the ones presented in \cite{BridgesBringmann}. We start with $\sum_4.$ We have three sums, the first two coming from $g_1(q)$ and the last one from $g_2(q).$

\begin{align*}
S_{41} &\coloneqq \sum_{\substack{0\leq h < k \leq N \\ \gcd(h,k)=1 \\ 4 | k}} \frac{\omega_{h,\frac{k}{2}}^4}{\omega_{h,\frac{k}{4}}^2} (-1)^{\frac{k}{2}+1} e^{\pi i \left(\frac{h'}{2}-\frac{3h'k}{4}\right)-\frac{2\pi inh}{k}} \int_{-\vartheta_{h,k}'}^{\vartheta_{h,k}''}e^{\frac{2\pi n z}{k}}g_1(q_1) d\Phi, \\
S_{42} &\coloneqq \sum_{\substack{0\leq h < k \leq N \\ \gcd(h,k)=1 \\ 4 | k}} \frac{\omega_{h,\frac{k}{2}}^4}{k\omega_{h,\frac{k}{4}}^2} e^{-\frac{2\pi inh}{k}} \sum_{\nu \pmod{k}} (-1)^{\nu} e^{\frac{\pi i h'(-3\nu^2+\nu)}{k}} \int_{-\vartheta_{h,k}'}^{\vartheta_{h,k}''} e^{\frac{2\pi nz}{k}} e^{-\frac{\pi}{12kz}} \xi(q_1)zI_{k,\nu}(z) d\Phi, \\
S_{43} &\coloneqq \sum_{\substack{0\leq h < k \leq N \\ \gcd(h,k)=1 \\ 4 | k}} \frac{\omega_{h,\frac{k}{2}}^6}{\omega_{h,k}^4\omega_{h,\frac{k}{4}}^2} e^{-\frac{2\pi inh}{k}} \int_{-\vartheta_{h,k}'}^{\vartheta_{h,k}''} e^{\frac{2\pi nz}{k}} g_2(q_1) d\Phi.
\end{align*}

We start by estimating $S_{41}.$ We use the splittings \eqref{Splitting} and \eqref{Splitting2} to decompose $S_{41}$ into three sums $S_{41}^{[1]}, S_{41}^{[2]}$ and $S_{41}^{[3]}.$ We find

$$S_{41}^{[1]} = \sum_{\substack{1 \leq k \leq N \\ 4|k}} (-1)^{\frac{k}{2}+1} \sum_{\substack{0\leq h < k\\ \gcd(h,k)=1}} \frac{\omega_{h,\frac{k}{2}}^4}{\omega_{h,\frac{k}{4}}^2} e^{\frac{\pi i h'}{2}\left(1-\frac{3k}{2}\right)-\frac{2\pi i nh}{k}} \int_{-\frac{1}{k(k+N)}}^{\frac{1}{k(k+N)}} e^{\frac{2\pi nz}{k}} \sum_{m \geq 0} a(m) e^{\frac{2\pi i m}{k}\left(h' + \frac{i}{z}\right)} d\Phi$$

$$ = \sum_{\substack{1 \leq k \leq N \\ 4|k}} (-1)^{\frac{k}{2}+1} \sum_{m \geq 0} a(m) K_k^{[41]}(n,m) \int_{-\frac{1}{k(k+N)}}^{\frac{1}{k(k+N)}} e^{\frac{2\pi}{k}(nz-\frac{m}{z})} d\Phi.$$

Recall that $\re{z}= \frac{k}{N^2}$ and from \eqref{BoundOnInverseOfZ} we infer $\re{z^{-1}} \geq \frac{k}{2}.$ Together with Lemma \ref{BoundsOnKloostermanSumsCase4|K} this implies the bound 

$$S_{41}^{[1]} \ll \sum_{1 \leq k \leq N} \frac{n^{\frac{1}{3}}}{k^{\frac{1}{3}-\varepsilon}(k+N)} \rightarrow 0, \quad (N \rightarrow \infty).$$

For $S_{41}^{[2]}$ we obtain with the same calculations 

$$S_{41}^{[2]} = \sum_{\substack{1 \leq k \leq N \\ 4|k}} (-1)^{\frac{k}{2}+1} \sum_{m \geq 0} a(m) \sum_{\ell = N+1}^{k+N-1} \mathbb{K}_{k,\ell}^{[41]}(n,m) \int_{-\frac{1}{k\ell}}^{-\frac{1}{k(\ell+1)}} e^{\frac{2\pi}{k}\left(nz-\frac{m}{z}\right)} d\Phi.$$

Bounding this as before with Lemma \ref{BoundsOnKloostermanSumsCase4|K} shows $S_{41}^{[2]} \rightarrow 0$ as $N \rightarrow \infty.$ The contribution $S_{41}^{[3]}$ is bounded in the exact same way as $S_{41}^{[2]}.$

Next, we are going to bound $S_{42}.$ We invoke the splittings \eqref{Splitting} and \eqref{Splitting2} to decompose $S_{42}$ into the three sums $S_{42}^{[1]},S_{42}^{[2]}$ and $S_{42}^{[3]}.$ We have as before

$$S_{42}^{[1]} = \sum_{\substack{1 \leq k \leq N \\ 4|k}} \frac{1}{k} \sum_{m \geq 0} r(m) \sum_{\nu \pmod{k}} (-1)^{\nu} K_k^{[42]}(\nu,n,m) \int_{-\frac{1}{k(k+N)}}^{\frac{1}{k(k+N)}} e^{\frac{2\pi}{k}\left(nz-\frac{m}{z}\right)}e^{-\frac{\pi}{12kz}}zI_{k,\nu}(z) d\Phi.$$

\noindent
We use as before the bounds $\re{z} = \frac{k}{N^2}$ and $\re{z^{-1}} \geq \frac{k}{2}$ as well as Lemma \ref{BoundsOnKloostermanSumsCase4|K} and Lemma \ref{BoundIntegral} to obtain as $N \rightarrow \infty,$

\begin{align*}
S_{42}^{[1]} &\ll \sum_{\substack{1 \leq k \leq N \\ 4|k}} \frac{1}{k} e^{\frac{2\pi n}{N^2}} \sum_{\nu \pmod{k}} n^{\frac{1}{3}}k^{\frac{2}{3}+\varepsilon} \int_{-\frac{1}{k(k+N)}}^{\frac{1}{k(k+N)}} \left|\mathcal{J}_{-\frac{1}{12},k,\nu}(z)\right| d\Phi \\
&\ll n^{\frac{1}{3}} \sum_{1 \leq k \leq N} \frac{k^{-\frac{1}{3}+\varepsilon}}{k(k+N)} \sum_{\nu \pmod{k}} \frac{1}{\left|\frac{\pi}{2} - \frac{\pi}{k}\left(\nu - \frac{1}{6}\right) \right|} \ll \frac{n^{\frac{1}{3}}}{N} \sum_{k=1}^{N} k^{-\frac{1}{3}+\varepsilon} \log(k)\\
&\ll n^{\frac{1}{3}}N^{-\frac{1}{3}+\varepsilon}\log(N) \rightarrow 0.
\end{align*}

\noindent
We bound $S_{42}^{[2]}$ in the same way by replacing the Kloosterman sum $K_k^{[42]}$ with the sum of the Kloostermans sums $\mathbb{K}_{k,\ell}^{[42]}(\nu,n,m).$ The sum $S_{42}^{[3]}$ is bounded exactly like $S_{42}^{[2]}.$

Finally, the contribution $S_{43} \rightarrow 0$ as $N \rightarrow \infty$ by the same argument as used for $S_{41},$ changing the Kloosterman sums $K_k^{[41]}$ and $\mathbb{K}_{k,\ell}^{[41]}$ to the Kloosterman sums $K_k^{[43]}$ and $\mathbb{K}_{k,\ell}^{[43]}.$ In total, we obtain that $\sum_{4} \rightarrow 0,$ as $N\rightarrow \infty.$

We continue to estimate $\sum_{2}.$ Thus, suppose that $\gcd(4,k)=2.$ We decompose $\sum_{2}$ into three sums

\begin{align*}
S_{21} &\coloneqq \frac{1}{4} \sum_{\substack{0\leq h < k \leq N \\ \gcd(h,k)=1 \\ \gcd(4,k)=2}} \frac{\omega_{h,\frac{k}{2}}^4}{\omega_{2h,\frac{k}{2}}^2} (-1)^{\frac{k}{2}+1}e^{\frac{\pi ih'}{2}\left(1-\frac{3k}{2}\right)-\frac{2\pi inh}{k}} \int_{-\vartheta_{h,k}'}^{\vartheta_{h,k}''} e^{\frac{2\pi nz}{k}+\frac{\pi}{2kz}} \frac{P(q_1^2)^4}{P(q_1)^3}f(q_1)d\Phi,\\
S_{22} &\coloneqq \frac{1}{2} \sum_{\substack{0\leq h < k \leq N \\ \gcd(h,k)=1 \\ \gcd(4,k)=2}} \frac{\omega_{h,\frac{k}{2}}^4}{k\omega_{2h,\frac{k}{2}}^2} e^{-\frac{2\pi inh}{k}} \sum_{\nu \pmod{k}} (-1)^{\nu} e^{\frac{\pi ih'(-3\nu^2+\nu)}{k}} \int_{-\vartheta_{h,k}'}^{\vartheta_{h,k}''} e^{\frac{2\pi nz}{k}+\frac{5\pi}{12kz}} \frac{P(q_1^2)^4}{P(q_1)^3} zI_{k,\nu}(z)d\Phi,\\
S_{23} &\coloneqq \frac{1}{4} \sum_{\substack{0\leq h < k \leq N \\ \gcd(h,k)=1 \\ \gcd(4,k)=2}} \frac{\omega_{h,\frac{k}{2}}^6}{\omega_{h,k}^4\omega_{2h,\frac{k}{2}}^2}e^{-\frac{2\pi inh}{k}} \int_{-\vartheta_{h,k}'}^{\vartheta_{h,k}''} e^{\frac{2\pi nz}{k}+\frac{\pi}{2kz}} \frac{P(q_1^2)^6}{P(q_1)^6} d\Phi.
\end{align*}

Eventhough, on a first glance $S_{21}$ and $S_{23}$ individually make a significant contribution to the overall sum, their principal parts do in fact cancel each other. This follows from the fact that they grow with the same factor of $e^{\frac{\pi}{2kz}}$ and 

\begin{equation}\label{MultiplierSumGCD(4,k)=2}
\frac{\omega_{h,\frac{k}{2}}^4}{\omega_{2h,\frac{k}{2}}^2} (-1)^{\frac{k}{2}+1}e^{\frac{\pi ih'}{2}\left(1-\frac{3k}{2}\right)} =- \frac{\omega_{h,\frac{k}{2}}^6}{\omega_{h,k}^4\omega_{2h,\frac{k}{2}}^2}.
\end{equation}

To see this, rewrite the equation as 

$$(-1)^{\frac{k}{2}}e^{\frac{\pi ih'}{2}\left(1-\frac{3k}{2}\right)} = \frac{\omega_{h,\frac{k}{2}}^2}{\omega_{h,k}^4}.$$

\noindent
Since $\gcd(4,k)=2$ we use the definition of $\omega_{h,k}$ and obtain that

$$\frac{\omega_{h,\frac{k}{2}}^2}{\omega_{h,k}^4} = e^{\frac{\pi i}{4}(4-h'k+h^2h'k-h(2+k))},$$

\noindent
which further simplifies, by using $hh' \equiv -1 \pmod{16k},$ to

$$=(-1)e^{-\frac{\pi i}{4}(h(2+k))}.$$

\noindent
Using again that $\gcd(4,k)=2$ and $hh' \equiv -1 \pmod{16k}$ we recognize $(-1)=(-1)^{\frac{k}{2}}$ and furthermore $e^{-\frac{\pi i}{4}(h(2+k))} = e^{\frac{\pi ih'}{2}\left(1-\frac{3k}{2}\right)},$ thus proving \eqref{MultiplierSumGCD(4,k)=2}. 

We now continue to estimate $S_{22}.$ The non-principal part vanishes as for $S_{42}$ in the limit $N\rightarrow \infty.$ We are left with 

$$\mathcal{S}_{22} \coloneqq\frac{1}{2} \sum_{\substack{0\leq h < k \leq N \\ \gcd(h,k)=1 \\ \gcd(4,k)=2}} \frac{\omega_{h,\frac{k}{2}}^4}{k\omega_{2h,\frac{k}{2}}^2} e^{-\frac{2\pi inh}{k}} \sum_{\nu \pmod{k}} (-1)^{\nu} e^{\frac{\pi ih'(-3\nu^2+\nu)}{k}} \int_{-\vartheta_{h,k}'}^{\vartheta_{h,k}''} e^{\frac{2\pi nz}{k}} \mathcal{J}_{\frac{5}{12},k,\nu}(z) d\Phi.$$

By Lemma \ref{BoundIntegral} (ii) we can write $\mathcal{J}_{\frac{5}{12},k,\nu}(z) = \mathcal{J}_{\frac{5}{12},k,\nu}^{*}(z) + \mathcal{E}_{\frac{5}{12},k,\nu},$ where the contribution from $\mathcal{E}_{\frac{5}{12},k,\nu}$ vanishes as before when $N \rightarrow \infty.$ We decompose $\mathcal{S}_{22}$ into $\mathcal{S}_{22}^{[1]}, \mathcal{S}_{22}^{[2]}$ and $\mathcal{S}_{22}^{[3]}$ using the splittings \eqref{Splitting} and \eqref{Splitting2}. The estimate of $\mathcal{S}_{22}^{[1]}$ below in particular is essentially the same as in \cite{BridgesBringmann} based on the ideas of Rademacher \cite{RademacherJFunction}.

Recall that $z=k(N^{-2}-i\Phi).$ We have for $b>0$

$$\int_{-\frac{1}{k(k+N)}}^{\frac{1}{k(k+N)}} e^{\frac{2\pi nz}{k}}\mathcal{J}_{b,k,\nu}^{*}(z) d\Phi = \int_{-\frac{1}{k(k+N)}}^{\frac{1}{k(k+N)}} e^{\frac{2\pi nz}{k}} \sqrt{\frac{b}{3}}\int_{-1}^{1} \frac{e^{\frac{\pi b}{kz}(1-x^2)}}{\cosh\left(\frac{\pi i}{k}\left(\nu-\frac{1}{6}\right)-\frac{\pi}{k}\sqrt{\frac{b}{3}}x\right)} dx d\Phi$$

$$=\frac{\sqrt{b}}{ik\sqrt{3}} \int_{-1}^{1}\frac{1}{\cosh\left(\frac{\pi i}{k}\left(\nu-\frac{1}{6}\right)-\frac{\pi}{k}\sqrt{\frac{b}{3}}x\right)} \int_{\frac{k}{N^2}-\frac{i}{k+N}}^{\frac{k}{N^2}+\frac{i}{k+N}} e^{\frac{\pi b}{kz}(1-x^2) +\frac{2\pi nz}{k}} dzdx.$$

After the change of variables $w=\frac{z}{k}$ the last expression equals
$$=\frac{\sqrt{b}}{i\sqrt{3}} \int_{-1}^{1}\frac{1}{\cosh\left(\frac{\pi i}{k}\left(\nu-\frac{1}{6}\right)-\frac{\pi}{k}\sqrt{\frac{b}{3}}x\right)} \int_{\frac{1}{N^2}-\frac{i}{k(k+N)}}^{\frac{1}{N^2}+\frac{i}{k(k+N)}} e^{2\pi nw + \frac{2\pi}{k^2w}\frac{b}{2}(1-x^2)} dwdx.$$

Thus, we find that

\begin{align*}
\mathcal{S}_{22}^{[1]} &= \frac{\sqrt{5}\pi}{6} \sum_{\substack{1\leq k \leq N \\ \gcd(4,k)=2}} \frac{1}{k} \sum_{\nu \pmod{k}} (-1)^{\nu}K_k^{[22]}(\nu,n)\int_{-1}^{1} \frac{L_k\left(n,\frac{5}{24}\left(1-x^2\right)\right)}{\cosh\left(\frac{\pi i}{k}\left(\nu-\frac{1}{6}\right)-\frac{\pi}{k}\sqrt{\frac{5}{36}}x\right)} dx\\
&- \frac{\sqrt{5}\pi}{6} \sum_{\substack{1\leq k \leq N \\ \gcd(4,k)=2}} \frac{1}{k} \sum_{\nu \pmod{k}} (-1)^{\nu}K_k^{[22]}(\nu,n)\int_{-1}^{1} \frac{1}{\cosh\left(\frac{\pi i}{k}\left(\nu-\frac{1}{6}\right)-\frac{\pi}{k}\sqrt{\frac{5}{36}}x\right)} \\
&\times\left(\mathcal{E}_k^{[1]}\left(n,\frac{5}{24}\left(1-x^2\right)\right)+\mathcal{E}_k^{[2]}\left(n,\frac{5}{24}\left(1-x^2\right)\right)+\mathcal{E}_k^{[3]}\left(n,\frac{5}{24}\left(1-x^2\right)\right)\right) dx,
\end{align*}

\noindent
where we denote by the $R$ the rectangle with edges $\pm \frac{1}{N^2} \pm \frac{i}{k(k+N)}$ around $0$ with counterclockwise orientation and set

\begin{align*}
L_k(n,y) &:= \frac{1}{2\pi i} \int_{R} e^{2\pi nw + \frac{2\pi y}{k^2w}}dw, \quad \mathcal{E}_k^{[1]}(n,y) := \frac{1}{2\pi i} \int_{\frac{1}{N^2}+\frac{i}{k(k+N)}}^{-\frac{1}{N^2}+\frac{i}{k(k+N)}} e^{2\pi nw + \frac{2\pi y}{k^2w}}dw, \\
\mathcal{E}_k^{[2]}(n,y) &:= \frac{1}{2\pi i} \int_{-\frac{1}{N^2}+\frac{i}{k(k+N)}}^{-\frac{1}{N^2}-\frac{i}{k(k+N)}} e^{2\pi nw + \frac{2\pi y}{k^2w}}dw, \quad \mathcal{E}_k^{[3]}(n,y) := \frac{1}{2\pi i} \int_{-\frac{1}{N^2}-\frac{i}{k(k+N)}}^{\frac{1}{N^2}-\frac{i}{k(k+N)}} e^{2\pi nw + \frac{2\pi y}{k^2w}}dw.	
\end{align*}

We first bound $\mathcal{E}_k^{[1]}$ and $\mathcal{E}_k^{[3]}.$ On the horizontal edges of $R$ we have, by estimates of Rademacher \cite{RademacherJFunction}

$$ w = u \pm \frac{i}{k(k+N)}, \quad -\frac{1}{N^2} \leq u \leq \frac{1}{N^2}, \quad \re{w} = u \leq \frac{1}{N^2}, \quad \re{\frac{1}{w}} \leq 4k^2.$$ 

Hence,

$$\left|\mathcal{E}_k^{[1]}\left(n,\frac{5}{24}\left(1-x^2\right)\right)\right|, \left|\mathcal{E}_k^{[3]}\left(n,\frac{5}{24}\left(1-x^2\right)\right)\right| \leq \frac{1}{N^2\pi}e^{\frac{5\pi}{3}\left(1-x^2\right)+\frac{2\pi n}{N^2}}.$$

Following again Rademacher \cite{RademacherJFunction} we have on the vertical part of $R$

$$ w = -\frac{1}{N^2} +iv, \quad -\frac{1}{k(k+N)} \leq v \leq \frac{1}{k(k+N)}, \quad \re{w}<0, \quad \re{\frac{1}{w}} < 0.$$

Thus,

$$\left|\mathcal{E}_k^{[2]}\left(n,\frac{5}{24}\left(1-x^2\right)\right)\right| \leq \frac{1}{\pi kN}.$$

These estimates together with Lemma \ref{BoundsOnKloostermanSumsCaseGCD(4,k)=2} imply $\mathcal{E}_k^{[1]},\mathcal{E}_k^{[2]}$ and $\mathcal{E}_k^{[3]}$ contribute at most

$$\ll e^{\frac{2\pi n}{N^2}} \sum_{k=1}^{N} \frac{1}{k} \sum_{\nu = 1}^{k} k^{\frac{2}{3}+\varepsilon}n^{\frac{1}{3}} \int_{-1}^{1} \frac{1}{\left|\cosh\left(\frac{\pi i}{k}\left(\nu-\frac{1}{6}\right)-\frac{\pi}{k}\sqrt{\frac{5}{36}}x\right)\right|} \frac{1}{kN} e^{\frac{5\pi}{3}\left(1-x^2\right)} dx.$$

\noindent
Following \cite{BridgesBringmann} we have for $\alpha \geq 0$ and $0<\beta < \pi$ the bound

$$|\cosh(\alpha +i \beta)| \geq \left|\sin\left(\frac{\pi}{2}-\beta\right)\right| \gg \left|\frac{\pi}{2}-\beta\right|.$$

\noindent
Since $\nu-\frac{1}{6} > 0$ (choosing here the representative system $\nu = 1 \dots k$ for $\nu \pmod{k}$) we can bound the above as $N\rightarrow \infty$ by

\begin{align*}
&\ll \frac{n^{\frac{1}{3}}}{N} \sum_{k=1}^{N} k^{-\frac{4}{3}+\varepsilon} \sum_{\nu=1}^{k} \frac{1}{\left|\frac{\pi}{2}-\frac{\pi}{k}\left(\nu-\frac{1}{6}\right)\right|} \int_{-1}^{1} e^{\frac{5\pi}{3}\left(1-x^2\right)} dx \ll \frac{n^{\frac{1}{3}}}{N} \sum_{k=1}^{N} k^{-\frac{1}{3}+\varepsilon} \log(k)\\
&\ll n^{\frac{1}{3}}N^{-\frac{1}{3}+\varepsilon}\log(N) \rightarrow 0.
\end{align*}

We now see as in \cite{BridgesBringmann} that

$$L_k(n,y) = \frac{1}{k} \sqrt{\frac{y}{n}} I_1\left(\frac{4\pi \sqrt{ny}}{k}\right).$$

\noindent
Indeed, by the Residue Theorem we have

$$L_k(n,y) = \text{Res}_{w=0} \; e^{2\pi nw + \frac{2\pi y}{k^2w}}.$$

\noindent
Using the series expansion of the exponential function we find that 

\begin{align*}
\text{Res}_{w=0} \; e^{2\pi nw + \frac{2\pi y}{k^2w}} &= \sum_{m \geq 0} \frac{1}{m!(m+1)!} (2\pi)^{2m+1} \frac{n^my^{m+1}}{k^{2m+2}} \\
&= \frac{1}{k}\sqrt{\frac{y}{n}} \sum_{m \geq 0} \frac{1}{m!(m+1)!} (2\pi)^{2m+1}\frac{n^{m+\frac{1}{2}}y^{m+\frac{1}{2}}}{k^{2m+1}}.
\end{align*}

\noindent
We recognize the last sum as the series representation of the Bessel function \cite{Temme}

$$I_{\ell}(z) \coloneqq \sum_{m \geq 0} \frac{1}{m!\Gamma(m+\ell+1)}\left(\frac{x}{2}\right)^{2m+\ell},$$
which proves the claim. Therefore, as $N \rightarrow \infty$ we find

$$\mathcal{S}_{22}^{[1]} = \frac{5\pi}{12\sqrt{6n}} \sum_{\substack{k\geq 1 \\ \gcd(4,k)=2}} \frac{1}{k^2} \sum_{\nu \pmod{k}} (-1)^{\nu} K_k^{[22]}(\nu,n) \mathcal{I}_{\frac{5}{12},k,\nu}(n).$$

We continue to bound $\mathcal{S}_{22}^{[2]}.$ The contribution of $\mathcal{S}_{22}^{[3]}$ can be bounded in the exact same way. We find

\begin{align*}
\mathcal{S}_{22}^{[2]} &= \frac{1}{2}\sum_{\substack{1\leq k \leq N \\ \gcd(4,k)=2}} \frac{1}{k} \sum_{\nu \pmod{k}} (-1)^{\nu} \sum_{\ell = N+1}^{k+N-1} \mathbb{K}_{k,\ell}^{[22]}(v,n)\sqrt{\frac{5}{12\cdot 3}}\\
&\times \int_{-1}^{1} \frac{1}{\left|\cosh\left(\frac{\pi i}{k}\left(\nu-\frac{1}{6}\right)-\frac{\pi}{k}\sqrt{\frac{5}{36}}x\right)\right|} \int_{-\frac{1}{k\ell}}^{-\frac{1}{k(\ell+1)}} e^{2\pi nw + \frac{5\pi}{12k^2w}\left(1-x^2\right)}d\Phi dx.
\end{align*}

Following Rademacher \cite{RademacherJFunction} we bound

$$\re{2\pi nw + \frac{5\pi}{12k^2w}\left(1-x^2\right)} \leq \frac{2\pi n}{N^2} + \frac{5\pi}{3}\left(1-x^2\right).$$

Thus, using Lemma \ref{BoundsOnKloostermanSumsCaseGCD(4,k)=2}  we find as $N\rightarrow \infty,$ 

\begin{align*}
\mathcal{S}_{22}^{[2]} &\ll e^{\frac{2\pi n}{N^2}} \sum_{k=1}^{N} \frac{1}{k} \sum_{\nu=1}^{k} \sum_{\ell = N+1}^{k+N-1} k^{\frac{2}{3}+\varepsilon}n^{\frac{1}{3}} \int_{-1}^{1} \frac{e^{\frac{5\pi}{3}\left(1-x^2\right)}}{\left|\cosh\left(\frac{\pi i}{k}\left(\nu-\frac{1}{6}\right)-\frac{\pi}{k}\sqrt{\frac{5}{36}}x\right)\right|} dx \int_{-\frac{1}{k\ell}}^{-\frac{1}{k(\ell+1)}} d\Phi\\
&\ll n^{\frac{1}{3}} \sum_{k=1}^{N} k^{-\frac{1}{3}+\varepsilon} k\log(k) \frac{1}{k(N+k)} \ll n^{\frac{1}{3}}N^{-\frac{1}{3}+\varepsilon} \log(N) \rightarrow 0.
\end{align*}

Therefore, in total as $N\rightarrow \infty$ we find that 

$$\sum_2 = \frac{5\pi}{12\sqrt{6n}} \sum_{\substack{k\geq 1 \\ \gcd(4,k)=2}} \frac{1}{k^2} \sum_{\nu \pmod{k}} (-1)^{\nu} K_k^{[22]}(\nu,n) \mathcal{I}_{\frac{5}{12},k,\nu}(n).$$

Finally, we have to evaluate $\sum_{1}.$ Thus, we assume that $\gcd(4,k)=1.$ We have the following three sums

\begin{align*}
S_{11} &\coloneqq \sum_{\substack{0\leq h < k \leq N \\ \gcd(h,k)=1 \\ \gcd(4,k)=1}} \frac{\omega_{2h,k}^4}{\omega_{4h,k}^2} (-1)^{\frac{k-1}{2}} e^{\frac{3\pi ih'}{4k}-\frac{2\pi inh}{k}} \int_{-\vartheta_{h,k}'}^{\vartheta_{h,k}''} e^{\frac{2\pi nz}{k}-\frac{5\pi}{8kz}} \frac{P\left(q_1^{\frac{1}{2}}\right)^4}{P(q_1)P\left(q^{\frac{1}{4}}\right)^2} \cdot \omega\left(q^{\frac{1}{2}}\right) d\Phi, \\
S_{12} &\coloneqq \sum_{\substack{0\leq h < k \leq N \\ \gcd(h,k)=1 \\ \gcd(4,k)=1}} \frac{\omega_{2h,k}^4}{k\omega_{4h,k}^2} e^{-\frac{2\pi inh}{k}} \sum_{\nu \pmod{k}} (-1)^{\nu} e^{\frac{\pi ih'(-3\nu^2-\nu)}{k}} \int_{-\vartheta_{h,k}'}^{\vartheta_{h,k}''} e^{\frac{2\pi nz}{k}+\frac{\pi}{24kz}} \frac{P\left(q_1^{\frac{1}{2}}\right)^4}{P(q_1)P\left(q^{\frac{1}{4}}\right)^2} zI_{k,\nu}(z) d\Phi, \\
S_{13} &\coloneqq \sum_{\substack{0\leq h < k \leq N \\ \gcd(h,k)=1 \\ \gcd(4,k)=1}} \frac{\omega_{2h,k}^6}{\omega_{h,k}^4\omega_{4h,k}^2} e^{-\frac{2\pi inh}{k}} \int_{-\vartheta_{h,k}'}^{\vartheta_{h,k}''} e^{\frac{2\pi nz}{k}-\frac{\pi}{8kz}} \frac{P\left(q_1^{\frac{1}{2}}\right)^6}{P(q_1)^4P\left(q_1^{\frac{1}{4}}\right)^2}d\Phi.
\end{align*}

Similar to $S_{41}$ and $S_{43}$ we can show that as $N \rightarrow \infty$ we have $S_{11}\rightarrow 0$ and $S_{13} \rightarrow 0.$ Furthermore for $S_{12}$ the non-principal part vanishes as $N\rightarrow \infty$ and we are left with 

$$\mathcal{S}_{12} = \sum_{\substack{0\leq h < k \leq N \\ \gcd(h,k)=1 \\ \gcd(4,k)=1}} \frac{\omega_{2h,k}^4}{k\omega_{4h,k}^2} e^{-\frac{2\pi nh}{k}} \sum_{\nu \pmod{k}} (-1)^{\nu} e^{\frac{\pi ih'(-3\nu^2-\nu)}{k}} \int_{-\vartheta_{h,k}'}^{\vartheta_{h,k}''} e^{\frac{2\pi nz}{k}} \mathcal{J}_{\frac{1}{24},k,\nu} d\Phi.$$

\noindent
Similar to $\mathcal{S}_{22},$ using Lemma \ref{BoundsOnKloostermanSumsCaseGCD(4,k)=1}, one can show that as $N \rightarrow \infty$

$$\sum_{1}= \frac{\pi}{12\sqrt{6n}} \sum_{\substack{k\geq 1 \\ \gcd(4,k)=1}} \frac{1}{k^2} \sum_{\nu \pmod{k}} (-1)^{\nu} K_k^{[12]}(\nu,n) \mathcal{I}_{\frac{1}{24},k,\nu}(n).$$

Combining our results we obtain our main Theorem \ref{maintheorem}. It is easy to see that the summand $k=2$ dominates the rest of the sum and thus we get as a consequence the following asymptotics as $n \rightarrow \infty$

$$\overline{p_1}(n) \sim \frac{5\pi}{48\sqrt{6n}} \sum_{\nu =0,1} (-1)^{n+\nu} K_2^{[22]}(\nu,n)\mathcal{I}_{\frac{5}{12},2,\nu}(n).$$

\noindent
We can simplify the situation quite a bit, by recognizing that

$$K_2^{[22]}(0,n) = (-1)^n, \quad K_2^{[22]}(1,n) = (-1)^{n+1}.$$

\noindent
Hence, we see that this removes the twist by the root of unity introtuced at the beginning and we find

$$\overline{p_1}(n) \sim \frac{5\pi}{48\sqrt{6n}} \left(\mathcal{I}_{\frac{5}{12},2,0}(n)+\mathcal{I}_{\frac{5}{12},2,1}(n)\right), \quad (n\rightarrow \infty).$$

\noindent
In \cite{Mauth} the author used analytic methods to find a strong asymptotic formula for $\mathcal{I}_{\frac{1}{18},1,0}(n).$ The same methods can be applied here to find asymptotic formulas for $\mathcal{I}_{\frac{5}{12},2,0}(n)$ and $\mathcal{I}_{\frac{5}{12},2,1}(n)$ and conclude that the sum in Theorem \ref{maintheorem} actually converges. We leave out the details. This gives another proof of \eqref{Asymptotics} which was shown before by Bringmann, Holroyd, Mahlburg, and Vlasenko \cite{BringmannHolroydMahlburgVlasenko}.

\end{document}